\documentclass[12pt,leqno,twoside,a4paper]{amsart}

\usepackage{enumerate,amssymb}
\usepackage{a4wide}

\newtheorem{defi}{Definition}[section]
\newtheorem{coro}[defi]{Corollary}
\newtheorem{theo}[defi]{Theorem}
\newtheorem{prop}[defi]{Proposition}

\newtheorem{lemma}[defi]{Lemma}

\begin{document}
\title{On the intersection graphs of modules and rings}
\author{Jerzy Matczuk, Marta Nowakowska and
Edmund R. Puczy{\l}owski}

\address{Institute of Mathematics, University of Warsaw, Warsaw,
Poland} \email{jmatczuk@mimuw.edu.pl}
\address{Institute of Mathematics, University of Silesia,
Katowice, Poland} \email{mnowakowska@us.edu.pl}
\address{Institute of Mathematics, University of Warsaw, Warsaw,
Poland} \email{edmundp@mimuw.edu.pl}

\thanks{This research was supported by the Polish National Center of
Science Grant No DEC-2011/03/B/ST1/04893}

\begin{abstract}
We classify modules and rings with some specific properties of their intersection graphs. In particular, we describe rings with infinite intersection graphs containing maximal left ideals of finite degree. This answers a question raised in \cite{A2}. We also generalize this result to modules, i.e. we get the structure theorem of modules for which their intersection graphs are infinite and contain maximal submodules of finite degree. Furthermore we omit the assumption of   maximality of submodules and still get a satisfactory characterization of such modules. In addition we show that, if the intersection graph of a module is infinite but its clique number is finite, then the clique and chromatic numbers of the graph coincide. This fact  was known earlier only in some particular cases. It appears that such equality holds also in the complement graph.
\end{abstract}

\maketitle

\section*{Introduction}

There are many studies of various graphs associated to modules or rings (also to some other algebraic structures). Among them the most natural and important for the structure theory of modules and rings seem to be their   intersection graphs. The aim of the paper is to classify modules (and rings) in terms of some specific properties of their intersection graphs.

For a given module $M$ the \it intersection graph \rm  $G(M)$ is defined to be a simple graph (i.e. unweighted, undirected graph containing no graph loops or multiple edges), whose vertices are nontrivial  submodules of $M$ (i.e. distinct from 0 and $M$) and two distinct vertices are adjacent if and only if their intersection is nonzero. The order of $G(M)$, denoted by $|G(M)|$, is defined as the cardinality of the set of vertices of $G(M)$. For a given ring $R$, the intersection graph $G(R)$ of $R$ is defined as the intersection graph of  the left $R$-module $R$. We will also study the {\it complement graph} of $G(M)$, which is denoted by $G^c(M)$, i.e. the simple graph with the vertices those of $G(M)$ and such that any two distinct vertices of $G^c(M)$ are adjacent if and only if they are not adjacent in $G(M)$.

The intersection graphs of modules and rings have been investigated intensively by various authors (see for instance   \cite{A1, A2, C} and the references within). Our initial motivation for the studies presented in this paper was to answer a question raised in \cite{A2}. Namely, in \cite{A2}, the authors obtained a number of results  concerning a description of rings which contain a maximal left ideal $I$ such that $I$, regarded as a vertex of $G(R)$,  is of finite degree, i.e. there are only finitely many left ideals of $R$ which intersect $I$ nontrivially. However they did not get a complete description of such rings and left this as an open problem. This problem turns out to be quite nontrivial and inspiring. It motivates many further natural questions, which concern the structure of modules quite deeply.  It appears also  that some of methods we apply  can be extended much further and they can be used to improve many existing results. For instance we are able to get a satisfactory, as it seems, description of modules having infinite intersection graphs with a finite clique number. The partial known results concerning this topic (cf. \cite{A1}) did not suggest that such description would be possible. Applying this description we prove that the clique and chromatic numbers of such graphs coincide. This again improves known results on this topic. We also study the clique and chromatic numbers of the complement graph of $G(M)$ and generalize known results on them. Some other results are also presented.

The paper is organized as follows. In the preliminary section we fix the notation and recall some notions on graphs, modules and rings  which will be need in the sequel. We also collect here some basic properties of intersection graphs.

In Section 2 we characterize modules $M$  containing  maximal submodule which, as vertices of $G(M)$, has degree smaller than $|G(M)|$. This allows us, in Theorem \ref{auxiliary 3},  to describe modules  having infinitely many submodules and containing  a maximal submodule  of finite degree. Finally, in Theorem \ref{finite degree 2} we relax the maximality of the involved submodule.

In Section 3, applying results from Section 2, we answer the question quoted above. In Theorem \ref{theo1},  we describe all rings $R$ such that $G(R)$ is infinite and $R$ contains a maximal left ideal whose degree is smaller than $|G(R)|$.

The results from Section 2 are also applied in Section 4 where we get, in Theorem \ref{description of deg}, a description of  infinite intersection graphs of modules whose clique number is finite. We show that, for such graphs, the clique and chromatic numbers are equal.  We also prove that the clique and chromatic numbers coincide in the complement graphs.

In the last section we describe certain modules and rings with intersection graphs of particular shape.  We offer (cf. Theorem \ref{triangle free rings}) a complete description of  rings (and their graphs) which are triangle-free.

\section{Preliminaries}

Throughout the paper all rings are associative with unity 1 and all modules are unitary left modules. We follow standard terminology concerning rings, modules and graphs. However to avoid   possible misunderstanding we fix below some terminology and notation.

For a given ring $R$, the \textit{ prime radical} of $R$ is denoted by $\beta (R)$ and the \textit{ socle} of a module $M$ is denoted by $Soc(M)$.   The endomorphism ring of a module $M$ is denoted by $End(M)$. It is well known (Schur's lemma) that if  $M$ is a simple module then $End(M)$ is a division ring. The direct sum of modules is denoted by $\oplus$.

The following lemma plays an important role in our studies. We present its easy proof for the convenience of the reader.

\begin{lemma}\label{auxiliary 1} Let $M$ be a module such that  $M=S\oplus T$, where $S$ is a simple module. Let $\pi_S,\pi_T$ be the canonical projections of $M$ onto $S$ and $T$, respectively. If $N$ is a nonzero submodule of $M$ and $K=\pi_T(N)$, then:
\begin{itemize}
\item[$(i)$]$N\subseteq S\oplus K$ and $N\cap T\subseteq K$;
 \item[$(ii)$] Suppose that $N\cap T=N\cap S=0$. Then there is an isomorphism $f:S\rightarrow K$ such that $N=\{ s+f (s) \mid s\in S\}$. The map $s\mapsto  s+f(s)$ defines an isomorphism of $S$ and $N$. In particular $S\simeq K\simeq N$;
\item[$(iii)$] Suppose that $T$ is simple. Then the map associating to a given isomorphism $f\colon S\rightarrow T$ the submodule $\{ s+f(s) \mid s\in S\}$ gives a bijection between the set $Iso(S,T)$ of all isomorphisms from $S$ to $T$ and the set of all nonzero submodules $N$ of $S\oplus T$ such that $N\cap S=N\cap T=0$. In particular,  $S$ and $ T$  are not isomorphic if and only if the only nontrivial submodules of $S\oplus T$ are $S$ and $T$;

 \item[$(iv)$] Suppose that $S\simeq T$ and $f\colon S\rightarrow T$ is an isomorphism of $S$ and $T$. Then the map $g\mapsto  f\circ g$ gives a bijection between the set $Aut(S)$ of all automorphisms of $S$ and the set $Iso(S,T)$. In particular the cardinality of the set of all nontrivial submodules of $S\oplus T$ is equal to $|Iso(S,T)|+2=|End(S)|+1$, where $End(S)$ is the ring of endomorphisms of $S$.
     \end{itemize}
     \end{lemma}

\begin{proof}
$(i)$ Let $n=s+t\in N\subseteq S\oplus T$, for some $s\in S,\; t\in T$. Then $\pi_T(n)=t$, so $n\in S\oplus K$ i.e. $N\subseteq S\oplus K$.
Now assume that $n=s+t\in N\cap T$. Since $n, t\in T$, $s=n-t\in S\cap T  =0$.  Hence $n=t=\pi_T(n)\in K$ and  $N\cap T\subseteq K$ follows.

$(ii)$ Notice that, as $N\cap T=0$ and $S$ is simple, $\pi_S\colon N\rightarrow S$ is an isomorphism and its inverse $F $ defines  a homomorphism $f\colon S\rightarrow T$ such that $F(s)=s+f(s)$. As $N\cap S=0$, $f$ is a nonzero homomorphism, the simplicity of $S$ yields that $f$ is an isomorphism of $S$ and $K$. Now it is easy to complete the proof of $(ii)$.

$(iii) $ Let $N$ be a nonzero submodule of $M$ such that $N\cap S=0=N\cap T$. Then, as $T$ is simple, $K=\pi_T(N)=S$ and $(ii)$ implies that $N=\{s+f_N(s)\mid s\in S\}$ for a suitable $f_N\in Iso(S,T)$. Conversely, if $g\in Iso(S,T)$, then $N_g=\{s+g(s)\mid s\in S\}$ is a submodule of $M$ with $N_g\cap S=N_g \cap T=0$. Those two assignments give the desired correspondence.

The statement $(iv)$ is a direct consequence of $(iii)$.
\end{proof}

The following corollary to the above lemma, although elementary,  is essential for our considerations. It appears that  modules we are interested in are somehow related to the situation described in it.
\begin{coro}\label{basic ex}
Let  $M=S\oplus S$, where $S$ is a simple module. Then $|G(M)|=|End(M)|+1$. In particular, if the division ring $End(S)$ is infinite, then $M$ is a module which has infinitely many proper submodules and  a maximal submodule $S$ which does not contain  proper submodules.
\end{coro}

The Goldie dimension plays a quite important role in our studies. We shortly recall some notions and results related to that concept.

A submodule $N$ of a module $M$ is called {\it essential} if $N\cap K\neq 0$, for every nonzero submodule $K$ of $M$. A nonzero module  is called {\it uniform} if all its nonzero submodules are essential.

\textit{The Goldie dimension } $u(M) $ of $M$ is defined as the supremum of cardinalities of   index sets of direct sums of nonzero submodules contained in $M$.

It is said that a set $\{U_t\}_{t\in T}$ of submodules is a $u$\textit{-basis} of $M$ if all $U_t$'s are uniform, the sum $\sum_{t\in T}U_t$ is direct and is an essential submodule of $M$. It is known \cite{GP} that a module $M$ has a $u$-basis if and only if every nonzero submodule of $M$ contains a uniform submodule. Moreover the cardinality of every $u$-basis of $M$ is equal to $u(M)$.


Below we  recall some standard concepts on graphs, which we need in our considerations. Let $G$ be a given graph with the set of all vertices $V(G)$. Our studies mainly concern the following invariants of $G$:

$\circ$ The {\it degree} $deg(v)$ of a vertex $v$ of $G$ is the cardinality of the set of all vertices distinct from $v$ which are adjacent to $v$. The degree of $v$ in  the complement graph $G^c$ is denoted by $deg^c(v)$;

$\circ$ A {\it clique} of $G$ is a  subset $\mathcal C$ of $V(G)$ such that arbitrary two distinct vertices of $\mathcal C$ are adjacent. The {\it clique number} $\omega (G)$ of $G$ is  defined as $sup\{ |\mathcal C| \mid \mathcal C \mbox{ is a clique of }G\}$. The clique number of $G^c$ is denoted by $\omega^c(G)$;

$\circ$ We say that $G$ \textit{can be  colored by elements from the set } $T$, if there exists a surjective map $f:V(G) \rightarrow T$ such that $f(v)\neq f(w)$, for any   adjacent vertices $v, w\in V(G)$. The {\it chromatic number} $\chi (G)$ of $G$ is the cardinality of a set $T$ such that $G$ can be colored by elements from $T$ and $G$ cannot be colored by elements of any set of cardinality smaller than  $|T|$. It is clear that $\omega (G)\leq \chi (G)$. The chromatic number of $G^c$ is denoted by $\chi^c(G)$;

$\circ$ The \textit{girth} $gr(G)$ of the graph $G$ is  the length of a shortest cycle contained in the graph and  $gr(G)=\infty$, if $G$ does not contain any cycles.

In some results the following specific graphs $G$, which we identify by some symbols, will appear. To indicate that a graph is isomorphic to one of the graphs specified below, we just write that they are equal.

$\circ$  $G$ is {\it complete}, if $V(G)$ forms a clique of $G$. We denote it by $\mathcal K_\alpha$, where $\alpha=|G|$.

It is clear that, for a module $M$, the graph $G(M)$ is complete if and only if $M$ is a uniform module.

$\circ$ $G$ is {\it null}, if $G$ has no edges. We denote it by $\mathcal N_\alpha$, where $\alpha=|G|$.

Obviously, for a module $M$, if $G(M)$ is a null graph then all nontrivial submodules of $M$ are simple. This easily implies that $G(M)$ is a null graph if and only if either $M$ contains the unique nontrivial submodule or $M=S_1\oplus S_2$, for some simple submodules $S_1, S_2$ of $M$. In the former case $G(M)=\mathcal N_1=\mathcal K_1$ and in the latter,  $G(M)=\mathcal N_{|Iso (S_1, S_2)|+2}$, by Lemma \ref{auxiliary 1}.

$\circ$ $G$ is {\it star}, if there is $v\in V(G)$ adjacent to all vertices distinct from $v$ and different elements of $V(G)\setminus \{v\}$ are not adjacent. We denote that graph $G$ by $\mathcal S_\alpha$, if $|G|=\alpha$.

$\circ$ $G$ is {\it triangle-free}, if it does not have three distinct   vertices which are adjacent to each other  or, equivalently,   $\omega(G)\leq 2$. If $G$ is triangle-free then $gr(G)\geq 4$.

It is easy to see that, for a given module $M$, the graph $G(M)$ is a star graph if and only if either $M$ contains precisely two nontrivial submodules $N_1\subset N_2$ or $M$ contains a direct sum $S_1\oplus S_2$ of simple modules $S_1, S_2$ and $S_1\oplus S_2$ is the unique maximal submodule of $M$. In the former case $G(M)=\mathcal K_2=\mathcal S_2$ and in the latter $G(M)=\mathcal S_{|Iso(S_1, S_2)|+3}$, by Lemma \ref{auxiliary 1}.

\section{The structure of modules containing a submodule  of finite degree}

For a given nontrivial  submodule $N$ of a module $M$, the {\it degree} $deg(N)$ of $N$ is defined as the degree of $N$ considered as a vertex of $G(M)$, i.e. $deg(N)$ is the cardinality  of the set of all nontrivial submodules $K\neq N$ of $M$ such that $N\cap K\neq 0$. Note that, if $N$ is a submodule of $M$, then $|G(M)|=deg(N)+deg^c(N)+1$. We will use this equation often.

We begin with a characterization of modules containing a nontrivial submodule  of degree 0 or 1.

Note first that, if $N\in V(G(M))$ and $deg(N)=0$, then  both $N$ and $M/N$ are simple modules.  Thus, in the same time, $N$ is  a maximal and  a minimal submodule of $M$. This easily implies that $deg(N)=0$ if and only if $N$ is either the unique nontrivial submodule of $M$ or $N$ is a simple module and $M=N\oplus S$, for a suitable simple submodule $S$ of $M$. In the former case $G(M)=\mathcal K_1$ (then $deg(N)=deg^c(N)=0$) and in the latter, by  Lemma \ref{auxiliary 1},  $G(M)={\mathcal N}_{|Iso(N,S)|+2}$.

Suppose now that $N\in V(G(M))$ and $deg(N)=1$. Then there exists exactly one submodule $N_1$ of $M$ such that $N_1\cap N\neq 0$ and $N_1\neq N$. Clearly, either $N_1\subset N$ or $N\subset  N_1$. In the former case, a nonzero submodule $S$ of $M$ such that $N\cap S=0$ cannot exist, as otherwise $N_1\oplus S$ would be a nontrivial  submodule of $M$ distinct from $N$ and $N_1$ and having nonzero intersection with $N$. This means that $G(M)=\mathcal K_2$, provided $N_1\subset N$.  Suppose  $N\subset  N_1$. Then the  module $N$ has to be   simple and $G(M)=\mathcal K_2$, provided every nonzero submodule of $M$ has nonzero intersection with $N$. If $N\cap S= 0$, for a nonzero submodule $S$, then $(N\oplus S)\cap N\neq 0$ and $N\oplus S\neq M$, as  $N_1 \cap ( N\oplus S)=N$. Hence, the uniqueness of $N_1$ implies that $N_1=N\oplus S$ is the unique maximal submodule of $M$ and $S$ is a simple module. Consequently, applying Lemma \ref{auxiliary 1}, one gets that  $G(M)=\mathcal S_{|Iso(N, S)|+3}$ in this case.

On the other hand it is   clear that,  if $G(M)=\mathcal K_2$ or $G(M)=\mathcal S_\alpha$, then $M$ contains a nontrivial submodule $N$ with $deg(N)=1$. Reassuming, we have proved that: a module $M$ contains a submodule $N$ of degree 1  if and only if $G(M)$ is a star graph.

Now we pass to the study of modules $M$ containing a maximal submodule  of degree smaller than $|G(M)|$. The obtained results extend those in \cite{A1} and play a substantial role in our studies in the next sections (in particular to answer  the  question from \cite{A2}, which was mentioned   in the introduction).

We begin with a result which collects several properties of modules we are interested in.

\begin{prop}\label{auxiliary 2}

Let $T$ be a maximal submodule of a module $M$.
\begin{itemize}
\item[$(1)$] If  $deg(T)<deg^c(T)$ then:

\begin{itemize}

\item[$(i)$] $M=S\oplus T$, for a simple submodule $S$ of $M$;

\item[$(ii)$] $|End(S)|=deg^c(T)$;

\item[$(iii)$] $T$ contains an essential simple submodule $S'$ such that $S\simeq S'$ and $S'$ is the unique simple submodule of $T$;

\item[$(iv)$] If $N$ is a nontrivial submodule of $T$, then $T/N$ does not contain   submodules isomorphic to $S$.

\end{itemize}
\item[$(2)$] Suppose that conditions from $(i)$ to $(iv)$ of $(1)$ hold.  Then:
\begin{itemize}

\item[$(i)$] $Soc(M)=S'\oplus S$ is an essential submodule of $M$, where $S'$ is as in $(1)(iii)$;

\item[$(ii)$] If $N$ is a submodule of $M$ such that $N\cap T\neq 0$, then either $N$ is contained in $T$ or
$N=(N\cap T)\oplus S$;

\item[$(iii)$] $deg(T)=|G(M/S')|+1$;
\item[$(iv)$]  $deg(T)=2|G(T)|=2|G(T/S')|+2$.
\end{itemize}

\end{itemize}
\end{prop}

\begin{proof}
$(1)$ Suppose that $deg(T)<deg^c(T)$. If $deg(T)=0$, then the remarks proceeding the proposition yield that all conditions from $(i)$ to $(iii)$ are satisfied. From those remarks we also get that if $deg(T)=1$, then $deg^c(T)\leq deg(T)$, which is impossible.  Therefore we can assume that $deg(T)\geq 2$.  Since  $deg(T)<deg^c(T)$, we can pick a nonzero submodule $S$ of $M$ such that $S\cap T=0$. Maximality of $T$ implies that $M=S\oplus T$ and $S$ is   simple. In particular    $(1)(i)$ holds.

Since $deg^c(T) > deg(T)\geq 2
$, there exists a nonzero submodule $K$ of $M$ such that $K\neq S$ and $K\cap T=0$. Hence, by Lemma \ref{auxiliary 1}, $T$ contains a simple submodule $S'$ such that $S\simeq S'$.

Let $\mathcal F$ be the family of all simple submodules of $T$ isomorphic to $S$.
 First, making use of   Lemma \ref{auxiliary 1},  let us see  that the family of nonzero submodules of $M=S\oplus T$ with zero intersection with $T$ is equal to $\bigcup_{U\in \mathcal F}{\mathcal F_U}$, where $\mathcal F_U$ is the family of nontrivial submodules of $U\oplus S$, distinct from $U$. Moreover $|\mathcal F_U|=|End(S)|$. Indeed, take a nonzero submodule $N$ of $M$ such that $N\cap T=0$. Lemma \ref{auxiliary 1} shows that $N\subseteq U\oplus S$, where $U=\pi(N)$ and $U\simeq S$. Since $U\subseteq T$ and $N\cap T=0$, $N\not= U$ follows. This shows that any nonzero submodule $N$ of $M$ with $N\cap T=0$ belongs to $  \mathcal F_U$. Conversely, assume now that   $N\in \bigcup_{U\in \mathcal F}{\mathcal F_U}$. Then there exists a simple submodule $U$ of $M$ such that $U\subseteq T$, $N\not=U$ and $N\subseteq U\oplus S$. By modularity, we have $N\cap T\subseteq (U\oplus S)\cap T=(T\cap S)+U=U$. Since $U$ is simple, either $N\cap T=0$ or $N\cap T=U$. If the latter case holds, then $U\subseteq N$. Moreover Lemma \ref{auxiliary 1} implies that all submodules belonging to $\mathcal F_U$ are simple. So $N$ is simple and $U=N$, a contradiction.

Now we will prove that $|\mathcal{F}|=1.$ Assume that $\mathcal F$ contains more than one submodule. Then there is a bijection $f:{\mathcal F}\rightarrow \mathcal F$ such that $f(U)\neq U$, for all $U\in \mathcal F$. For each $U\in \mathcal F$ define $\mathcal G_U=\{ f(U)+N \mid N\in \mathcal F_U\}$. Note that $f(U)+N=f(U)\oplus N$, for all $N\in \mathcal F_U$. Indeed, assume that $f(U)\cap N\not=0$. Since $f(U)$ is simple, $f(U)=f(U)\cap N$. This means that $N\subseteq T$. However $N\in\mathcal F_U$, so $N\cap T=0$, a contradiction.

It is easy to check that, for distinct $N_1, N_2 \in \mathcal F_U$, also $f(U)+N_1$ and $f(U)+N_2$ are distinct. Hence $|\mathcal G_U|=|\mathcal F_U|$. Note that, for distinct $U_1, U_2\in \mathcal F$, the families $\mathcal G_{U_1}, \mathcal G_{U_2}$ are disjoint. Indeed, assume that $f(U_1)+N_1=f(U_2)+N_2$, where $N_i$ is a nontrivial submodule from $\mathcal F_{U_i}$ for $i=1,2$. By the definition of $\mathcal F_{U_i}$ we obtain that $T\cap N_i=0$. Now $f(U_1)=T\cap (f(U_1)+N_1)=T\cap (f(U_2)+N_2)=f(U_2)$, a contradiction. It is clear that, for every $U\in \mathcal F$, all modules in $\mathcal G_U$ have nonzero intersection with $T$. Consequently $deg^c(T) \leq deg(T)$. Hence $\mathcal F =\{S'\}$ and $deg^c(T)=|\mathcal F_{S'}|=|End(S)|$. Thus $(1)(ii)$ holds.

Now we will show that $S'$ is an essential submodule of $T$. Assume that there is a nonzero submodule $U$ of $T$ such that $S'\cap U=0$. Then $(U+S')\cap S=0$. This implies that $U\cap (S\oplus S')=0$. Now, if $K_1, K_2$ are distinct submodules of $S\oplus S'$, then $(K_1+U)\cap (S\oplus S')=K_1\neq K_2=(K_2+U)\cap (S\oplus S')$, so $K_1+U\neq K_2+U$. However $S\oplus S'$ contains $deg^c(T)$ nontrivial, distinct from $S'$, submodules and for each nonzero submodule $K$ of $S\oplus S'$ we have $0\neq U\subseteq (K+U)\cap T$. This contradicts the assumption that $deg(T)<deg^c(T)$. Hence $S'$ is an essential submodule of $T$ and $(1)(iii)$ follows.

Observe that Lemma \ref{auxiliary 1} together with the statement $(1)(ii)$ yield that, if $U$  is any module isomorphic to $S$,  then the cardinality of the set of all nontrivial  submodules of $U\oplus S$ distinct from $U$   is equal to $|End(S)|=deg^c(T)$. Thus, if for a nontrivial submodule $N$ of $T$, the module $T/N$ would contain a submodule isomorphic to $S$, then $T\oplus S$ would contain $deg^c(T)$ submodules containing $N$ and hence we would get that $deg^c(T) \leq deg(T)$, a contradiction. This proves $(1)(iv)$.

(2)  Since $M=T\oplus S$ and $S$ is a simple module, we have   $Soc(M)=Soc(T)\oplus S$. However, by $(1)(iii)$, $Soc(T)=S'$ is an essential submodule of $T$. This gives $(2)(i)$.

 Suppose that $N$ is a submodule of $M$ such that $N\cap T\neq 0$ and $N\not\subseteq T$. If $N\neq (N\cap T)\oplus S$, then obviously $N\cap T\neq T$ and $S\not\subseteq N$. Hence $N/(N\cap T)$ is a nonzero submodule of $(T/(N\cap T))\oplus S$ such that $(N/(N\cap T))\cap (T/(N\cap T))=0$ and $(N/(N\cap T))\cap S=0$. Applying Lemma \ref{auxiliary 1} we get that $T/(N\cap T)$ contains a submodule isomorphic to $S$, which contradicts $(1)(iv)$. This shows $(2)(ii)$.

$(2)(iii)$ is a direct consequence of $(1)(iii)$.  Finally, $(2)(iv)$ is a consequence of $(2)(ii)$ and $(1)(iii)$. This completes the  proof of the proposition.
\end{proof}

Now, with the help of the above proposition,  we are able to  characterize modules $M$ whose intersection graphs are infinite and they contain   maximal submodules of  degree smaller than $|G(M)|$.

\begin{theo}\label{auxiliary 3}

Suppose that $T$ is a maximal submodule of a module $M$. The following conditions are equivalent:

\begin{itemize}
\item[(1)] $|G(M)|$ is infinite and $deg(T)<|G(M)|$;
\item[(2)] The following conditions are satisfied:
\begin{itemize}

\item[$(i)$] $M=S\oplus T$, for a simple submodule $S$ of $M$;

\item[$(ii)$] $T$ contains an essential simple submodule $S'$ such that $S\simeq S'$;
\item[$(iii)$] $|End(S)|=|G(M)|$ is infinite;

\item[$(iv)$] $|G(M/S')|+1<|G(M)|$.
\end{itemize}
\end{itemize}
\end{theo}

\begin{proof} $(1)\Rightarrow (2)$  Suppose that $|G(M)|$ is infinite and $deg(T)< |G(M)|$. Since $|G(M)|=deg(T)+deg^c(T)+1$ we obtain that $deg(T)<deg^c(T)$. Thus, the assumptions of  Proposition \ref{auxiliary 2} are satisfied and due to statements (1) and $(2)(iii)$ of the proposition, it is enough to show that $|End(S)|=|G(M)|$.

 By Proposition \ref{auxiliary 2} $(1)(ii)$ we have that $deg^c(T)=|End(S)|$. Now, as  $|G(M)|$ is infinite, we have $|G(M)|=deg(T)+deg^c(T)+1=deg(T)+|End(S)|$. Moreover, as $deg(T)<|G(M)|$, the above equality reduces to  $|End(S)|=|G(M)|$.

$(2)\Rightarrow (1)$ Since $S'$ is a simple submodule of $T$, which is essential, we get that all submodules of $M$ with nonzero intersection with $T$ contain $S'$. Hence, by $(iv)$, $deg(T)< |G(M)|.$
\end{proof}

Obviously, the above results apply when $G(M)$ is infinite and $deg(T)$ is finite so we get the following corollary improving the results on this topic presented in \cite{A1, A2}.

\begin{coro}\label{finite degree}
Let $T$ be a maximal submodule of a module $M$. Then the following conditions are equivalent:
\begin{itemize}
\item[(1)] $|G(M)|$ is infinite and $deg(T)<\infty$;
\item[(2)] The following properties hold:
\begin{itemize}
\item[$(i)$] $M=S\oplus T$, for a simple submodule $S$ of $M$;

\item[$(ii)$]  $T$ contains the unique essential simple submodule $S'$ and $S'\simeq S$;

\item[$(iii)$] $End(S)$ is an infinite  division ring;

\item[$(iv)$] $G(M/S')$ is finite.
\end{itemize}
\end{itemize}
\end{coro}

     Observe that, due to Lemma \ref{auxiliary 1},    the condition $(iii)$ of the above corollary  can be replaced with the following condition expressed in the language of graphs: \\$(iii')$    $G(S\oplus S)$ is an infinite null graph.

 Notice  that if  $M$ is a module satisfying  the assumptions of Corollary  \ref{finite degree}, then the length $l(M)$ is finite and $|G(T)|<\infty$.
 In fact we have:
 \begin{prop}\label{l(M)}
  Suppose that a module $M$ contains a nonzero submodule $N$ of finite degree in $G(M)$, then $G(N)$ is finite and $M$ is of finite length.
 \end{prop}
 \begin{proof}
  Obviously, $|G(N)|< \infty$ and $|G(M/N)|<\infty$. In particular,  the modules $N$ and $M/N$ are of finite length and $l(M)=l(N)+l(M/N)$ is also finite, as required.
 \end{proof}

The above  suggests that it would be interesting to  characterize modules of finite length whose intersection graphs are infinite. The module $M$ from Corollary \ref{basic ex} is such and we will see in Theorem \ref{finite lenght} that  modules with the above property are exactly  the ones which have a homomorphic image  containing a submodule isomorphic to $M$.  For doing  so we will need the following lemma:

\begin{lemma}\label{auxiliary 4}
Let $T$ be a maximal submodule of a module $M$.  If $G(T)$ is finite and $G(M)$ is infinite, then:
 \begin{itemize}
   \item[(1)]  $T$ contains a submodule $N$ such that $G(M/N)$ is infinite and $deg(T/N) $ is finite in the graph $G(M/N)$;
   \item[(2)]  $M/N$ contains a submodule $S\oplus S'$, where $S$ and $S'$  are isomorphic simple submodules and $End(S)$ is an infinite division ring.
 \end{itemize}

\end{lemma}
\begin{proof}
Suppose that  $G(M)$ is infinite and $T$ is a maximal submodule of $M$ such that $G(T)$ is finite.

 (1) Let $\mathcal F$ be the family of all submodules $N$ of $T$ such that $N$ is contained in infinitely many submodules of $M$. Clearly, $0\in \mathcal F$, as $G(M)$ is infinite. Moreover $T\not\in \mathcal F$, by the maximality of $T$. Since $G(T)$ is finite, also the family $\mathcal F$ is finite. Thus  we can pick $N\in \mathcal F$, which is maximal in $\mathcal F$. It is clear that $G(M/N)$ is infinite and $G(T/N)$ is finite.

Now we will show that $deg(T/N)<\infty$. Assume that  there are infinitely many submodules $K/N$ of $M/N$ such that $(K/N)\cap (T/N)\neq 0$. Since $G(T/N)$ is finite, $T/N$ contains a nonzero submodule $L/N$ such that $(K/N)\cap (T/N)=L/N$, for infinitely many submodules $K$ of $M$. However this means that $L$ belongs to $\mathcal F$, which contradicts the maximality of $N$. Consequently $deg(T/N)<\infty$.

Obviously, $T/N$ is a maximal submodule of $M/N$. Thus,  applying Corollary \ref{finite degree} to $M/N$ and $T/N$, we get (2).
\end{proof}

\begin{theo}\label{finite lenght}
 For a module  $M$  of  finite length, the following conditions are equivalent:
 \begin{itemize}
   \item[(1)]  $G(M)$ is infinite;
   \item[(2)]   $M$ can be homomorphically mapped onto a module containing a submodule $S\oplus S'$, where $S$ and $S'$ are isomorphic simple submodules and $End(S)$ is an infinite division ring.
   \end{itemize}

\end{theo}

\begin{proof}
 The implication $(2)\Rightarrow (1)$ is an immediate consequence of Lemma \ref{auxiliary 1}.

$(1)\Rightarrow (2)$ Suppose that the length   of $M$ is finite and $G(M)$ is infinite. Consider the family $\mathcal F$ consisting of all  nonzero submodules  $N$ of $M$ such that $G(N)$ is finite. Since $M$ is of finite length, it contains simple submodules and clearly all of them are in $\mathcal F$. Thus $\mathcal F$ is non-empty. Moreover we can find a submodule $T$, which is maximal in $\mathcal F$. Obviously, $T\neq M$. Since $M$ is of  finite length we can pick a submodule $U$ of $M$ such that $T\subset U$ and $U/T$ is a simple module. By the maximality of $T$ in $\mathcal F$ we have that $G(U)$ is infinite, $T$ is a maximal submodule of $U$ and $G(T)$ is finite. Now, to complete the proof of the theorem,  it suffices to apply Lemma \ref{auxiliary 4} to $U$ and $T$.
\end{proof}

We close this section with a theorem  characterizing modules having infinite intersection graphs and containing a submodule of finite degree.

\begin{theo}\label{finite degree 2}
Let $N$ be a nontrivial submodule of a module $M$. The following conditions are equivalent:

\begin{itemize}
\item[$(1)$] $G(M)$ is infinite and $deg(N)<\infty$;
\item[$(2)$] The following conditions hold:
\begin{itemize}
\item[$(i)$] $N$ contains the unique simple submodule $S$;

\item[$(ii)$] $End(S)$ is infinite;

\item[$(iii)$] $G(M/S)$ is finite or, equivalently, $deg(S)<\infty$;

\item[$(iv)$] There exist submodules $B\subseteq A\subseteq M$ such that $N\cap A=0$ and $A/B\simeq S$.
\end{itemize}
\end{itemize}
\end{theo}

\begin{proof}
$(1)\Rightarrow (2)$ Suppose that $G(M)$ is infinite and let $N$ be a nontrivial submodule of  $M$ with $\deg(N) < \infty$. Then, by Proposition \ref{l(M)}, the length   of $M$ is finite and $N$ has only finitely many submodules. Let $\mathcal{F}$ be the family of submodules containing $N$ and having  only finitely many submodules. Since $N\in\mathcal{F}$ and $M$ is Noetherian,  there exists $N'\in\mathcal{F}$, which is maximal in $\mathcal{F}$. The infinity of  $G(M)$ guarantees that  $N'\neq M$.

Let $U$ be a submodule of $M$ containing $N'$ and such that $U/N'$ is a simple module. The choice of $N'$ yields that $G(U)$ is infinite.

Consider the family $\mathcal G$ of submodules of $N'$ which are contained in infinitely many submodules of $U$. Since $M$ is Noetherian, we can pick a maximal submodule $B$ of $\mathcal G$. Then, by the construction, there exists   a submodule $A$ of $U$ such that $A$ is not contained in $N'$ and $A\cap N'=B$. Then $N'+A=U$, as $U/N'$ is simple.  We will prove  that submodules $A$ and $B$ satisfy conditions of $(iv)$.

Since $deg(N)$ is finite, $N\cap B=0$. Hence $0=N\cap B=N\cap (N'\cap A)=N\cap A$. Note that $U/B=(N'+A)/B\simeq (N'/B)\oplus (A/B)$. In addition $U/N'=(A+N')/N'\simeq A/(A\cap N')=A/B$. Hence $A/B$ is a simple module and $N'/B$ is a maximal submodule of $U/B$.

Notice that $deg(N'/B) $ is finite in $G(U/B)$. Otherwise  there would be infinitely many submodules $X_t/B$ of $U/B$, where $t\in T$, which intersect $N'/B$ nontrivially. However $G(N')$ is finite, so there has to exist a submodule $C$ of $N'$ such that $N'\cap X_t=C$, for infinitely many $t\in T$. Obviously, $B\subset C$ and $C$ is a submodule of $N'$ contained in infinitely many submodules of $U$. This contradicts the maximality of $B$ in $\mathcal G$ and proves that $deg(N'/B) $ is finite in $G(U/B)$. By the construction,  $G(U/B)$ is infinite. We can now apply Corollary \ref{finite degree} to the module $U/B$ and its maximal submodule $N'/B$ to obtain an essential simple submodule $S_0$ of  $N'/B$  which is isomorphic to $A/B$ and has infinite endomorphism ring.

Recall that $N\cap B=0$. Thus, as $S_0$ is simple and essential in $U/B$,  $S_0\subseteq (N+B)/B\simeq N$. This implies that there exists the unique simple submodule $S$ of $N$ such that $A/B\simeq S$ and the division ring $End(S)$ is infinite.  This completes the proof of statements $(i), (ii), (iv)$.

 By $(i) $ we know  that every submodule $K$ of $M$ such that $K\cap N\neq 0$ contain $S$. Hence $|G(M/S)|\leq deg(N)<\infty$, i.e. the condition  $(iii)$ is also fulfilled.

$(2)\Rightarrow (1)$ Conditions  $(i)$ and $(iii)$ imply  that the module  $M$ has finite length and $deg(N)<\infty$.

The statement (2) implies that $End(S)$ is infinite, $S+B/B\simeq S\simeq A/B$ and the sum $(S+B)/B+A/B$ is direct. Now the infinity of  $G(M)$ is a direct consequence of  Theorem \ref{finite lenght}.
\end{proof}

\section{Classification of rings containing maximal left ideals of finite degree}

In this section we classify  all rings $R$ such that $|G(R)|$ is infinite and $R$ contains a maximal left ideal $T$ such that $deg(T)<|G(R)|$. This, in particular, answers a question from \cite{A2} about a characterization of rings containing  maximal left ideals of  finite degree.
To be more precise, this question as it is stated, concerns also rings $R$ with $G(R)$ finite. In \cite{H, S} it was proved that a ring has finitely many left ideals if and only if it is a direct product of a finite number of left Artinian serial rings (i.e. Artinian rings whose left ideals form a chain) and a finite ring. The ideal structure of serial rings is very simple, so it seems that this result is quite satisfactory from the graph theory point of view. A complete classification (up to isomorphism) of all Artinian serial rings is quite involving and treated in many papers, so it is rather hard to expect that one can classify such rings completely. Also a complete classification of simple rings is a rather  hopeless task. Thus it seems that the authors in \cite{A2} had in mind the case when $G(M)$ is infinite, which we treat here. Concerning serial rings and their complexity let us mention the recent paper \cite{JJ} in which there were described  non-commutative rings $R$ such that $\beta (R)$ is a simple $R$-module, $\beta (R)$ is contained in the center of $R$ and $R/\beta (R)$ is a field.

In the sequel we will need the following well known properties of minimal left ideals of rings (cf. \cite{J}).

\begin{lemma}\label{auxiliary 5} Let $R$ be a ring. Then:
\begin{itemize}
\item[$(i)$] (a special case of Nakayama's lemma) If $L$ is a minimal left ideal of $R$, then $\beta (R)L=0$ and $L$ is a simple left $R/\beta (R)$-module;

\item[$(ii)$] If $L$ is a minimal left ideal of  $R$ and $L^2\neq 0$, then $L$ contains an idempotent $e$ such that $L=Re$;

\item[$(iii)$] For given minimal left ideals  $S_1, S_2$ of  $R$ we have:
\begin{itemize}

\item[$(a)$] If $S_1S_2\neq 0$, then $S_1$ and $S_2$ are isomorphic as left $R$-modules;

\item[$(b)$] If $S_1, S_2$ are isomorphic as left $R$-modules then $S_1S_2=0$ if and only if $S_1^2=0$.
\end{itemize}
\end{itemize}

\end{lemma}

The following theorem answers the question from  \cite{A2}, which was discussed    at the beginning of the section.

\begin{theo}\label{theo1}
For a ring $R$, the following conditions are equivalent:
\begin{itemize}
\item [$(1)$] The graph $G(R)$ is infinite and $R$ contains a maximal left ideal $T$ satisfying $deg(T)<|G(R)|$;
\item[$(2)$] Up to isomorphism, $R$ is one of the following rings:
\begin{itemize}
\item[$(i)$]
$R=M_2(\Delta)$ the ring of all  $2$ by $2$ matrices over an infinite division ring $\Delta$. In this case $T$ can be any nontrivial left ideal of $R$ and $deg(T)=0$;

\item[$(ii)$]$R=\left(\begin{array}{*{2}{c}}
\Delta&\Delta\\0& P
\end{array}\right) \subseteq M_2(\Delta)$,
where $P$ is a subring of an infinite division ring $\Delta$ such that the cardinality of the set of left ideals of $P$ is smaller than $|\Delta|$. In this case $T=\left(\begin{array}{*{2}{c}}
0&\Delta\\0& P
\end{array}\right)$ is the unique maximal left ideal of $R$.
\end{itemize}
\end{itemize}
Moreover $R$ contains a maximal left ideal $T$ of finite degree and $G(R)$ is infinite if and only if either  $R$ is as in $(2)(i)$ (so $\deg(T)=0)$ or $R$ is as in $(2)(ii)$ and $P$ is a division subring of $\Delta$
(then $\deg(T)=2$).
\end{theo}
\begin{proof} $(1)\Rightarrow (2)$ Suppose that $R$ contains a maximal left ideal $T$ with $deg(T)<\infty$ and $G(R)$ is infinite.  Obviously, we can apply
the above obtained results on modules to the module $_RR$. By Theorem \ref{auxiliary 3} we have:

$-$ $R=S\oplus T$, where $S$ is a minimal left ideal of $R$;

$-$ $T$ contains the unique minimal left ideal $S'$ of $R$, which is an essential $R$-submodule of $T$ and $S\simeq S'$ as $R$-modules;

$-$ $\Delta=End(_RS)$ is a division ring and $|\Delta|=|G(R)|$ is infinite;

$-$ $Soc(R)=S\oplus Soc(T)=S\oplus S'$ is a two-sided ideal of $R$, which is  an essential left $R$-submodule (left ideal) of $R$.
\vspace{0,5mm}

If $\beta(R)=0$, then  $Soc(R)$ is a semiprime Artinian ring, so it is a ring with unity and is a direct summand of $R$. However $Soc(R)$ is essential in $R$, so $R=Soc(R)$ and $R\simeq End(_RR)=End(S\oplus S')\simeq M_2(\Delta)$.
Every nontrivial left ideal $T$ of $R$ is   maximal, so    $deg(T)=0$. This completes the proof in the case $\beta(R)=0$.

Assume now that $\beta (R)\neq 0$.
 We claim that
 $TS=0$ and $T$ is a two-sided ideal of $R$.
  Assume that $TS\not=0$. Since $S$ is a minimal left ideal of $R$, there is  $s\in S$ such that $Ts=S$. This implies that $S\simeq T/N$, where $N$ is the kernel of the $R$-module epimorphism of $T$ onto $S$ defined by $t\mapsto ts$. Recall that, as $\deg(T)<|G(R)|$, the assumptions of   Proposition \ref{auxiliary 2} are fulfilled and the statement $(1)(iv)$ of the proposition forces $N=0$. Then $T$ would be equal to $S'$ and $R=S\oplus S'$ would be a semisimple ring, which is impossible as  $\beta (R)\neq 0$. This
   proves the claim.

 Now, the  rings $R/T$ and $S$ are isomorphic. Moreover, as  $S$ is a minimal left ideal of $R$,  $R/T$ is a ring  with unity containing no nontrivial left ideals. Thus $S\simeq R/T$ is a division ring. This fact together  with $TS=0$ imply that
   $S\simeq End(_SS)=End(_RS) =\Delta$ is an infinite division ring.
   We will identify $S$ and $\Delta$ using the above isomorphism.  Let $e$ be the
identity element of $S$. Then $e$ is an idempotent of $R$ such that
$S=Re=eRe$ and $T=R(1-e)$. Moreover $S'=T\cap Soc(R)$ is the intersection of two-sided ideals of $R$, so $S'$ itself is a two-sided ideal. Hence $eR(1-e)=eReR(1-e)=ST\subseteq T\cap Soc(R)=S'$. By Lemma \ref{auxiliary 5} $(iii)(b)$ one gets that $ST\neq 0$. Therefore, as $S'$ is a minimal left ideal of $R$,  $eR(1-e)=S'$ follows. Thus we know that
 $S=eR=eRe$ is a division ring with identity $e$ and $S'=eR(1-e)$.

Let $ Q=(1-e)R(1-e)$.
Fix a nonzero element $ m\in eR(1-e)$. It is clear, as $eS'=S'$ is a simple $R$-module and $S$ is a division ring, that $Sm=S'$ and $p'm\ne 0$, for any nonzero $p'\in S$.
Notice also that $mp\in S'=Sm$, for any $p\in Q=(1-e)R(1-e)$. The above implies that, for any $p\in Q$,  there exists uniquely determined element  $f(p)\in S=eRe$ such that $mp=f(p)m$.   It is easy to check that
$f:Q\rightarrow S$ is a ring embedding and $P=f(Q)$ is a domain, as   a subring of the division ring $S$.

 Since $T=R(1-e)=Q\oplus S'$ and   $S'$ is   a two-sided ideal of $R$, we obtain $T/S'$ is isomorphic to $Q$. Hence there exists a bijection between the set of left ideals of rings $P$ and $T/S'$. Applying Theorem \ref{auxiliary 3} and the fact that $G(R)$ is infinite, one gets that the cardinality of the set of left ideals of $P$ is less than $|G(R)|=|End(_RS)|=|\Delta|$.

 Now consider the Pierce decomposition $R=eRe\oplus eR(1-e)\oplus (1-e)R(1-e)=S\oplus S'\oplus Q$ of $R$ with respect the idempotent $e$.  Then every  $r\in R$ can be
uniquely presented in the form $r=d_1+d_2m+f^{-1}(p)$ for some $d_1,d_2\in S=
\Delta$ and $p\in P=f(Q)$. Take another element $r'=d_1'+d_2'm+f^{-1}(p')\in R$ , where $d_1', d_2'\in \Delta$ and $p'\in P$. Then $rr'=d_1d_1'+d_1d_2'm+d_2mf^{-1}(d_1') +f^{-1}(p)f^{-1}(p')= d_1d_1'+(d_1d_2'+d_2d_1')m+f^{-1}(pp')$. Now it is clear the map $\Psi : R \rightarrow \left(\begin{array}{*{2}{c}}
\Delta&\Delta\\0& P
\end{array}\right)$ defined by the formula $\Psi (d_1+d_2m+f^{-1}(p))=\left(\begin{array}{*{2}{c}}
d_1&d_2\\0& p
\end{array}\right)$ is a ring isomorphism of $R$ onto the subring $\left(\begin{array}{*{2}{c}}
\Delta&\Delta\\0& P
\end{array}\right)$ of $M_2(\Delta)$. Note that this isomorphism maps $T$ onto $U=\left(\begin{array}{*{2}{c}}
0&\Delta\\0& P
\end{array}\right)$.
It is not hard to check that $U$ is the only maximal left ideal of $M=\left(\begin{array}{*{2}{c}}
\Delta&\Delta\\0& P
\end{array}\right)$ such that $deg(U)<|G(M)|$.

Assume now that  $deg(T)<\infty$. Then, by Proposition \ref{l(M)}, $R$ is an Artinian ring. Hence
$P\simeq Q  \simeq R/(S\oplus S')$ is an Artinian domain, i.e. $P$
  is a division ring. In this case there are precisely two nontrivial left ideals distinct from $T$, which have nonzero  intersection with $T$, namely $S\oplus S'$ and $S'$. Thus   $deg(T)=2$,   in this case.

$(2)\Rightarrow (1)$ It is clear that if $R$ is isomorphic to $M_2(\Delta)$, where $\Delta$ is an infinite  division ring, then $G(R)$ is infinite and every nontrivial left ideal $T$ of $R$ is a maximal left ideal of $R$ such that $deg(T)=0$.

Assume now that $R$ is as described in $(2)(ii)$ and $T=\left(\begin{array}{*{2}{c}}
0&\Delta\\0& P
\end{array}\right)$.  It is easy to see that  $T$ is a maximal left ideal of $R$ and the nontrivial left ideals of $R$, which are distinct from $T$ and have nonzero intersection with $T$ are precisely ideals of the type $\left(\begin{array}{*{2}{c}}
0&\Delta\\0& L
\end{array}\right)$ and $\left(\begin{array}{*{2}{c}}
\Delta&\Delta\\0& L
\end{array}\right)$, where $L\neq P$   is a left ideal of $P$. The set of  nontrivial left ideals of $R$, which have 0 intersection with $T$ is equal to $\{L_d\mid d\in \Delta\}$, where  $L_d =\{   \left(\begin{array}{*{2}{c}}
x&xd\\0& 0
\end{array}\right) \mid x\in \Delta\} $. Therefore $|G(R)|=|\Delta|$ and  $ \deg  T=2\mathfrak{n}$, where $\mathfrak{n}$ denotes the cardinality of the set of all  left ideals $L$ of $P$, $L\ne P$. By assumption $|\Delta |$ is infinite and  $2\mathfrak{n}<|\Delta |$. Thus $G(R)$ is infinite and     $deg(T)<|G(R)|$. Moreover $\deg T=2$ if and only if $\mathfrak{n}=1$, i.e. $P$ is a division ring. This completes the proof.
\end{proof}

\section{Modules with intersection graphs of finite clique number}


Recall that  $\omega(G(M))$ denotes the clique number of the intersection graph $G(M)$. We will write $\omega(M)$ instead of $\omega(G(M))$.

In the following proposition we collect several properties of modules whose intersection graphs are infinite and have finite clique number. Those properties form a background for a description of all such modules and their graphs.

\begin{prop}\label{auxiliary 6} Let  $M$ be a module such that $G(M)$ is infinite and $\omega (M)<\infty$. Then:

\begin{itemize}
\item[$(i)$] The length $l(M)$ of $M$ is finite and $Soc(M)$ is an essential submodule of $M$;

\item[$(ii)$] There exist simple, isomorphic submodules $S$ and $S'$ of $M$ such that $End(S)$ is an infinite division ring and  $Soc(M)=S\oplus S'$;

\item[$(iii)$] If $N_1, N_2$ are nontrivial submodules of $M$ which do not contain $Soc(M)$, then either $N_1\cap N_2=0$ or $N_1\cap Soc(M)=N_2\cap Soc(M)$ is a simple submodule of $M$;

\item[$(iv)$] Every nontrivial submodule of $M$ either contains $Soc(M)$ or is uniform;

\item[$(v)$] $\mathcal C$ is a maximal clique of $G(M)$ if and only if there is a simple submodule $N$ of $M$ such that ${\mathcal C}={\overline N}$, where ${\overline N}$ is the set of all nontrivial submodules of $M$ containing $N$.
    \end{itemize}
    \end{prop}

\begin{proof} $(i)$  It is clear that any chain  of nontrivial submodules of $M$  forms  a clique. Thus $l(M)\leq \omega(M)$ is finite. Now every nonzero submodule of $M$ contains a simple submodule, so $Soc(M)$ is essential in $M$ and $(i)$ follows.

$(ii)$ Let $\mathcal R$ be the family of nonzero submodules $U$  of $M$ such that $G(U)$ is finite. Since $l(M)<\infty$, $\mathcal R$ is non-empty (it contains all simple  submodules of $M$) and we can find a submodule $X$, which is maximal in $\mathcal R$. By assumption   $G(M)$ is infinite, so  $X\neq M$. Let $L$ be a submodule of $M$ containing $X$ and such that $L/X$ is a simple module. From the choice of $X$ it follows that $L$ contains infinitely many submodules. By Theorem \ref{finite lenght}, there exists a homomorphic image of $L$, say $L/K$, containing a submodule $S\oplus S'$ such that $S$ and $S'$ are isomorphic simple modules and $End(S)$ is an infinite division ring. Then, by Lemma \ref{auxiliary 1}, $L/K$ contains infinitely many submodules $L_t/K$, where $t$ ranges over some infinite set $T$. Notice that $K$ has to be the zero submodule, as otherwise the modules $L_t$ would  form an infinite clique but we know that $\omega (M)$ is finite.
 Therefore $M$ contains a submodule $S\oplus S'$ such that $S$ and $S'$ are isomorphic simple modules and $End(S)$ is an infinite division ring.

 If  $V$ is a simple submodule of $M$ such that $(S\oplus S')\cap V=0$, then the set $\{U\oplus V \mid U \mbox{ is a nonzero submodule of }S\oplus S'\}$ forms an infinite clique of $G(M)$, a contradiction. This implies that $Soc(M)$  is equal to  $S\oplus S'$ and is an essential submodule of $M$, i.e.   $(ii)$ holds.

$(iii)$ Suppose  that  $N_1, N_2$ are nontrivial submodules of $M$ not containing $Soc(M)$. Then $(ii)$ and Lemma \ref{auxiliary 1} imply that $N_1\cap Soc(M)$ and $N_2\cap Soc(M)$ are simple submodules of $M$. Consequently, if $N_1\cap Soc(M)\neq N_2\cap Soc(M)$, then $N_1\cap N_2\cap Soc(M)=0$. Applying $(ii)$ once again we get that $N_1\cap N_2=0$. Hence $(iii)$ follows.

 $(iv)$   Let us observe that, if  $N_1, N_2$ are nonzero submodules of $M$ and $N_1\cap N_2=0$, then $(ii)$ and the fact that $l(Soc(M))=2$ give that $Soc (M)=(N_1\cap Soc(M))+(N_2\cap Soc (M))\subseteq N_1+N_2$. The above shows that if a  submodule of $M$  is not uniform, then it has to contain  $Soc(M)$. This gives $(iv)$.

$(v)$ Let $N$ be a simple submodule of $M$. It is clear that the set $\overline N$ defined   in $(v)$ is a clique of $G(M)$.  The simplicity of $N$ yields that, for any submodule $K\not\in \overline N$, $K\cap N=0$.  Hence $\overline{N}$ is a maximal clique.

Conversely, suppose that $\mathcal C$ is a maximal clique of $G(M)$.  The statement $(iii)$ implies that there exists a simple submodule $N$ of $M $ contained in every submodule from $\mathcal C$. Therefore $\mathcal C\subseteq \overline N$ and maximality of $\mathcal C$   gives   ${\mathcal C}=\overline N$. This completes the proof.
\end{proof}

\begin{theo}\label{description of deg} For a  module $M$ the following conditions are equivalent:

\begin{itemize}
\item[$(1)$] $G(M)$ is infinite and $\omega (M)$ is finite;
\item[$(2)$] The following conditions hold:
\begin{itemize}

\item[$(i)$] $Soc(M)=S\oplus S'$, where $S$ and $S'$   are isomorphic simple modules such that $End(S)$ is an infinite division ring;

\item[$(ii)$] All but finitely many submodules of $M$ are contained in $Soc(M)$.
    \end{itemize}
    \end{itemize}
\end{theo}

\begin{proof}$(1)\Rightarrow (2)$ Suppose that $G(M)$ is infinite and $\omega (M)$ is finite. The condition $(i)$ was  proved in Proposition \ref{auxiliary 6}, so is is enough to  prove $(ii)$.

For a given submodule $V$ of $M$, the set  of all submodules containing $V$ forms a clique. In particular every such set is finite, as $\omega(M)$ is such.
Assume that the family $\mathcal F$ of submodules of $M$ not contained in $Soc(M)$ is infinite. The above  implies that $\mathcal F$ has to contain an infinite family $\mathcal F_1$ such that, for arbitrary distinct submodules $A,B\in \mathcal F_1$, $A\cap Soc(M)\neq B\cap Soc(M)$.  Proposition \ref{auxiliary 6} $(iii)$ says that in this case we also have  $A\cap B=0$. Since $l(Soc(M))=2$ and $A\cap Soc(M)$ and $B\cap Soc(M)$ are distinct nonzero submodules of $Soc(M)$, we get that  $(A\cap Soc(M))+(B\cap Soc(M))=Soc(M)$. Note that $A+ Soc(M)\neq B+ Soc(M)$, as the equality $A+ Soc(M)=B+ Soc(M)$ would imply that on one hand $ (A+Soc(M))\cap B=(B+Soc(M))\cap B=B$ and on the other $(A+Soc(M))\cap B=(A+(B\cap Soc(M)))\cap B=(A\cap B)+(B\cap Soc(M))=B\cap Soc(M)\neq B$, as by the assumption $B\not\subseteq Soc(M)$. This shows that the set $\{ F+Soc(M) \mid F\in \mathcal F_1\}$ is an infinite clique in $G(M)$, a contradiction. This implies that $\mathcal{F}$ has to be finite, i.e. $(ii)$ holds.

$(2)\Rightarrow (1)$ Suppose that the conditions $(i)$ and $(ii)$ are satisfied. Then Lemma \ref{auxiliary 1} implies that $G(M)$ is infinite.

By Proposition \ref{auxiliary 6}(v), any maximal clique is of the form $\overline N$, where $N$ is a certain simple submodule of $M$. Hence, the property (ii) together with the fact that the only submodules of $Soc(M)$ containing $N$ are just $Soc(M)$ and $N$ itself yield that  $\omega(M)$ is finite.
\end{proof}

Proposition \ref{auxiliary 6} and Theorem \ref{description of deg} allow to describe the structure of $G(M)$ when $G(M)$ is infinite and $\omega(M)$ is finite. This is quite simple. Namely the set of vertices of $G(M)$ is the union of the maximal cliques $\overline N$, where $N$ runs over the family $\mathcal S$ of all nontrivial submodules of $Soc(M)$ (the cardinality of $\mathcal S$ is equal to $|End(S)|+1$, where $S$ is as in Proposition \ref{auxiliary 6}). Let $\mathcal L$ be the family of all proper submodules of $M$ containing $Soc(M)$. Obviously, $\mathcal L$ is finite. By Theorem \ref{description of deg} there are only finitely many nontrivial submodules $N$ of $Soc(M)$ such that ${\overline N}\neq \{ N\} \cup {\mathcal L}$. For all $N\in \mathcal S$, $|{\overline N}|\leq \omega (M)$. Moreover,
 for arbitrary distinct $N_1, N_2\in \mathcal S$, ${\overline N_1}\cap {\overline N_2}=\mathcal L$. Hence all sets ${\overline N}\setminus {\mathcal L}$ are disjoint.\vspace{2mm}

Applying the above remarks we can  get the following theorem. Recall that for a graph $G$,  $\chi(G)$ denotes its chromatic number.  We will write $\chi(M)$ for $\chi(G(M))$.

\begin{theo}\label{chromatic}
Let $M$ be a module such that $G(M)$ is infinite and $\omega (M)<\infty$, then $\omega (M)=\chi (M)$.
\end{theo}

\begin{proof} We follow the notation proceeding the theorem.  We can pick, by  Proposition \ref{auxiliary 6}(v), a submodule  $N\in\mathcal{S}$ such that $\omega (M)=|{\overline N}|$. Then  $|{\overline K}|\leq |{\overline N}|$, for all $K\in \mathcal S$.

   Let us color all vertices in $\overline N$ by distinct colors. Take an arbitrary $N_1\neq N\in \mathcal S$. Then ${\overline N}\cap {\overline N_1}=\mathcal L$ and $|{\overline N}\setminus {\mathcal L}|\geq |{\overline N_1}\setminus {\mathcal L}|$. Color all submodules in ${\overline N_1}\setminus {\mathcal L}$ by arbitrary distinct colors used to color submodules in ${\overline N}\setminus {\mathcal L}$. This is possible as $|{\overline N}\setminus {\mathcal L}|\geq |{\overline N_1}\setminus {\mathcal L}|$. By Proposition \ref{auxiliary 6}(iii), for arbitrary distinct  $N_1, N_2\in \mathcal S$ and $V_1\in {\overline N_1}\setminus {\mathcal L}$, $V_2\in {\overline N_2}\setminus {\mathcal L}$, we have $V_1\cap V_2=0$. Therefore in this way  we have obtained a proper vertex coloring with $|{\overline N}|=\omega (M)$ colors. This forces the equality of  $\chi(M)$ and $\omega (M)$, as $\chi (M)\geq \omega(M)$ always.

\end{proof}

Let us observe that if $M$ is a module such that $G(M)$ is infinite and $\omega (M)<\infty$ then, by Theorem \ref{description of deg}, $M$ contains infinitely many simple submodules. In particular   $\omega(M)\leq\omega^c(M)$. The following result says that this   inequality  yields    equality  $\omega^c (M)=\chi^c (M)$  in slightly more general  setting. Namely:

\begin{theo}\label{dual clique-chromatic}

Let  $M$ be a module with a  $u$-basis (i.e. every nonzero submodule contains a uniform submodule).  Then $\omega^c (M)=\chi^c (M)$, provided $\omega(M)\leq \omega^c(M)$.
\end{theo}

\begin{proof} Applying Zorn's lemma we can find, among cliques of $G(M)$ consisting of uniform modules, a maximal one. Say ${\mathcal C}=\{ U_t\}_{t\in T}$ is such a clique of $G(M)$.  We can assume that   $T$ is a  well-ordered set.

 Let $V(M)$ denote the set of all nontrivial submodules of $M$. By assumption $M$ has a $u$-basis. Thus  every $N\in V(M)$ contains an uniform submodule and the choice of  $\mathcal C$ implies that there is $t\in T$ such that $N\cap U_t\neq 0$. Therefore we can define a map  $f\colon V(M)\rightarrow T$ by setting
 $f(N)=\min\{s\in T\mid N\cap U_s\ne 0\}$.  Suppose that $f(N_1)=f(N_2)$. Then $U_{f(N_1)}\cap N_1\neq 0\neq U_{f(N_2)}\cap N_2=U_{f(N_1)}\cap N_2$. Since $U_{f(N_1)}$ is uniform, $U_{f(N_1)}\cap N_1\cap N_2\neq 0$ follows. Hence $N_1\cap N_2\neq 0$ and $N_1, N_2$ are not adjacent in $G^c(M)$. Consequently $T$ gives a proper coloring of $G^c(M)$, so   $\chi^c(M)\le |T|$.
   Therefore, if  $\omega(M)\leq\omega^c(M)$, we obtain $\omega^c(M)\leq \chi^c(M)\leq |T|=|\mathcal{C}| \leq \omega(M)\leq \omega^c(M)$. This gives the required equality.
 \end{proof}
The above theorem, the  remark preceding it and Proposition \ref{auxiliary 6} $(iv)$ give  immediately the following corollary:
\begin{coro}
 Let $M$ be a module such that $G(M)$ is infinite and $\omega (M)<\infty$. Then $\omega^c (M)=\chi^c (M)$.
\end{coro}

\section{On some properties of the intersection graphs of modules}

We conclude the paper with   some basic and rather straightforward properties of   intersection graphs of modules and rings, completing some of those in \cite{A1, A2}.

Recall that a graph $G(M)$ is triangle-free if it does not have three distinct vertices which are adjacent to each other  or, equivalently,   $\omega(M)\leq 2$.

\begin{prop}\label{triangle free} For a   module $M$, the following conditions are equivalent:
\begin{itemize}
  \item[(1)]  The graph $G(M)$ is triangle-free;
  \item[(2)] $M$ is   one of the following independent types:
  \begin{itemize}
    \item[$(i)$]  $M$ is a chain module of length $l(M)\leq 3$;
    \item[$(ii)$]  $M=Soc(M)=S_1\oplus S_2$, for some simple submodules $S_1, S_2$ of $M$;
       \item[$(iii)$]$ Soc(M)=S_1\oplus S_2$, is the unique maximal submodule of $M$, where $S_1$ and $ S_2$ are simple submodules of $M$;
  \end{itemize}
   \item[(3)]  $G(M)$ is one of the following graphs:
 $\mathcal N_0,\;\mathcal K_1=\mathcal N_1,\; \mathcal K_2=\mathcal S_2, \;\mathcal N_\alpha,\; \mathcal S_\alpha$, where $\alpha=|G(M)|\geq 3$.
 \end{itemize}

\end{prop}

\begin{proof} $(1)\Rightarrow (2)$ Note that, if a module $M$ contains a direct sum $N_1\oplus N_2\oplus N_3$ of three nonzero submodules $N_1, N_2, N_3$, then the submodules $N_1\oplus N_2, N_2\oplus N_3, N_1\oplus N_3$ form a triangle.  Moreover any chain of submodules of $M$ consisting of three nontrivial submodules forms a triangle in $G(M)$. Consequently, if $G(M)$ is triangle-free, then  the Goldie dimension $u(M)$ of $M$ is not greater than $2$ and $l(M)\leq 3$. Now it is clear that, if $M$ is a uniform module, then $M$ is as described in $(2)(i)$.

Suppose now  that   $u(M)=2$. Since $l(M)\leq 3$, $Soc(M)=S_1\oplus S_2$, for some simple submodules $S_1, S_2$ of $M$. If $M=Soc(M)$, then $(2)(ii)$ holds.

 Assume that $M\ne Soc(M)$.  Let $N$ be a proper submodule of $M$ different from $S_1\oplus S_2$. Considering the set of submodules $\{N\cap (S_1\oplus S_2),S_1\oplus S_2,N\}$ and using that fact that  $G(M)$ is triangle-free, we obtain $N=N\cap (S_1\oplus S_2)\subset  S_1\oplus S_2$. This shows that $S_1\oplus S_2$ is the only maximal submodule of  $M$, i.e. $M$ is as in $(2)(iii)$.

The implication $(2)\Rightarrow (3)$ is a direct consequence of  Lemma \ref{auxiliary 1}, the implication $(3)\Rightarrow (1)$ is clear.
\end{proof}

Recall that $gr(G(M))$ denotes the girth of the graph $G(M)$ and when $G(M)$ is triangle-free then $gr(G(M))\geq 4$.

Proposition \ref{triangle free} gives immediately the following improvement of Theorem 2.5 in \cite{A1}.

\begin{coro} Let  $M$ be a module such that its graph $G(M)$ is triangle-free.  Then $G(M)$ contains no cycles, so   $gr(G(M))=\infty$.
\end{coro}

   Some results on rings $R$ with  triangle-free graphs $G(R)$  were obtained in Theorem 17 of \cite{A2}. In the following theorem we  classify such graphs completely and describe the respective rings.

\begin{theo}\label{triangle free rings} For a ring $R$ the following conditions are equivalent:

\begin{itemize}
  \item[(1)]  $G(R)$ is triangle-free;
  \item[(2)]  $R$ is one of the following independent types:
  \begin{itemize}
  \item[$(i)$]  Either $R$ is a division ring or $R$ contains precisely one nontrivial left ideal (equal to the prime radical   $\beta (R)$) or $R$ contains precisely two left ideals $\beta (R)$, $(\beta (R))^2$. Moreover, the graph $G(R)$ is equal to $\mathcal{N}_0$, $\mathcal{K}_1$ and $\mathcal{K}_2$, respectively;

\item[$(ii)$] $R$ is isomorphic either to $M_2(\Delta)$ or to $\Delta\oplus \Delta$ where $\Delta$ is a division ring. In the former case $G(R)={\mathcal N_{|\Delta |+1}}$ and in the latter $G(R)={\mathcal N_2}$;

    \item[$(iii)$] $R/\beta (R)=\Delta$ is a division ring, $(\beta (R))^2=0$  and $\beta(R)=S_1\oplus S_2$, where $S_1, S_2$ are minimal left ideals both isomorphic to $\Delta$ as left $R$-modules (and also as $\Delta$-modules). In this case $G(M)={\mathcal S_{|\Delta|+2}}$.
        \end{itemize}
\end{itemize}

        \end{theo}

\begin{proof}

$(1)\Rightarrow (2)$ We apply Proposition \ref{triangle free} and classic results on Artinian rings.

If $_RR$ is as described in Proposition \ref{triangle free}$(2)(i)$, then  it is clear that the ring $R$ has to be one of the rings described in $(2)(i)$.

Let $_RR$ be as described in Proposition \ref{triangle free}$(2)(ii)$. Then      $R$ is a semisimple ring and Wedderburn-Artin theorem implies that $R$ has to be as in $(2)(ii)$.

 Finally, assume that  $_RR$ is as described in Proposition \ref{triangle free}$(2)(iii)$.
Thus there exist minimal left ideals $S_1,S_2$ of $R$ such that $Soc(R)=S_1\oplus S_2$ is the unique maximal left ideal of $R$. In particular,  $\Delta=R/Soc(R)$ is a division ring.
    Note that none of $S_1, S_2$ can contain a nonzero idempotent. Indeed, if $e$ were a nonzero idempotent in $S_1$, then we would have $S_1=Re$. Hence  $R(1-e)$ would be a maximal left ideal of $R$ different from $S_1\oplus S_2$, which is impossible. Therefore $S_1^2=S_2^2=0$ and, consequently,  $S\beta(R)=S_1\oplus S_2$. Clearly, $\beta (R)S_1=\beta (R)S_2=0$. This  implies that $_\Delta S_1\simeq {_\Delta\Delta}\simeq { _\Delta S}_2$, so also $_RS_1\simeq {_R}S_2$. Moreover, using Lemma \ref{auxiliary 1}$(iv)$, one can see that $G(M)={\mathcal S_{|Iso(S_1, S_2)|+2+1}}={\mathcal S_{|\Delta|+2}}$, as the rings $End(_RS_1), End(_\Delta S_1), \Delta$ are isomorphic.

      The implication $(2)\Rightarrow (1)$ is a direct  consequence of Proposition \ref{triangle free}.
\end{proof}

For a given module $M$, if $G(M)$ is not a null graph, then $M$ contains nontrivial submodules $N_1\subset  N_2$. Now if $N$ is a submodule  of $M$, which is maximal with respect to $N_1\cap N=0$ (it may happen that $N=0$), then $N_1\oplus N$ is a nontrivial  essential submodule of $M$. Then each nontrivial submodule of $M$ is adjacent to $N_1\oplus N$. Consequently arbitrary two distinct nontrivial submodules of $M$ are connected via $N_1\oplus N$ (obviously, they can also be adjacent to each other). Consequently $G(M)$ is connected if and only if $M$ is not the direct sum of two simple modules and in this case the diameter of $G(M)$ is not greater than $2$ (cf. \cite{A1}, Theorems 2.1 and 2.4).


\begin{thebibliography}{99}


\bibitem{A1} S. Akbari, H. A. Tavallaee and S. Khakashi Ghezelahmad,  Intersection graph of submodules of a module, J. Algebra Appl. 11 (2012) 1250019.

\bibitem{A2} S. Akbari, R. Nikandish and J. Nikmehr, Some results on the intersection graphs of ideals of rings, J. Algebra Appl. 12 (2013) 1250200.

\bibitem{C} I. Chakrabarty, S. Ghosh, T. K. Mukherjee and M. K. Sen, Intersection graphs of ideals of rings, Discrete Math. 309 (2009), 5381--5392.

\bibitem{GP} P. Grzeszczuk and E. R. Puczy\l owski, On the infinite Goldie dimension of modular lattices and modules, J. Pure Appl. Algebra 35 (1985), 151--155.


\bibitem{H} Y. Hirano, Rings with finitely many orbits under the regular action, Rings, modules, algebras, and abelian groups, 343–-347, Lecture Notes in Pure and Appl. Math., 236, Dekker, New York, 2004.

\bibitem{J} N. Jacobson, { \it Structure of rings}, Amer. Math. Soc. Coll. Publ., 37, Providence, 1964.

\bibitem{JJ} J. Jelisiejew, On commutativity of ring extensions, Comm.  Alg. 44 (2016), 1931--1940.



\bibitem{O} J. Olszewski, On ideals of products of rings, Demonstratio Math. 27 (1994),  1--7.



\bibitem{S} W. Stephenson, Modules whose lattice of submodules is distributive, Proc. London Math. Soc. 28 (1974), 291–-310.

\end{thebibliography}
\end{document}